\documentclass[12pt]{article}

\usepackage{amssymb,amsmath,amsfonts,amsthm}
\usepackage{graphicx}
\usepackage{epstopdf}
\usepackage{caption}
\usepackage[dvipsnames,usenames]{color}
\usepackage{float}
\usepackage{subfigure}
\usepackage[toc,page]{appendix}

\newtheorem{theorem}{Theorem}
\newtheorem{corollary}{Corollary}

\newtheorem{lemma}{Lemma}
\newtheorem{remark}{Remark}

\usepackage{hyperref}
\usepackage{authblk}

\begin{document}

\numberwithin{equation}{section}
\numberwithin{theorem}{section}

\title{The M/G/1 retrial queue with event-dependent arrivals}
\author{Ioannis Dimitriou}
\affil{Department of Mathematics, University of Patras, 26504 Patras, Greece}
\affil[ ]{\textit {E-mail: idimit@math.upatras.gr}}
\renewcommand\Authands{ and }
\providecommand{\keywords}[1]{\textbf{\textit{Keywords---}} #1}
\maketitle

\begin{abstract}
We introduce a novel single-server queue with general retrial times and event-dependent arrivals. This is a versatile model for the study of service systems, in which the server needs a non-negligible time to retrieve waiting customers upon a service completion, while future arrivals depend on the last realized event. Such a model is motivated by the customers' behaviour in service systems where they decide to join based on the last realized event. We investigate the necessary and sufficient stability condition and derive the stationary distribution both at service completion epochs, and at an arbitrary epoch using the supplementary variable technique. We also study the asymptotic behaviour under high
rate of retrials. Performance measures are explicitly derived and extensive numerical examples are performed to investigate the impact of event-dependency. Moreover, constrained optimisation problems are formulated and solved with ultimate goal to investigate the admission control problem. 
\end{abstract}


\section{Introduction}
In this work, we introduce a novel queueing system with a special feature for the customers’ behaviour, called \textit{event-dependency}, in the presence of retrials. In many service systems, arriving customers choose to get served remotely. Thus, in case they find an idle server they begin their service immediately; otherwise, they register to the system and their demands are queued at a virtual queue, called the orbit queue. When the server becomes idle, he/she turns his/her attention to the orbit and seeks for a customer to begin a new service. This seeking time is usually non-negligible, since the retrieval of a customer from the orbit usually requires some preliminary steps (e.g., to call back
the customer). If a newly arriving customer arrives during the seeking time, the retrieval
process is cancelled, and the server starts serving the newly arriving customer. 

Our aim is to study a versatile model for the representation of such service systems, by introducing an \textit{event-dependent} queueing model with a general retrial/retrieval policy. Our motivation stems also from the fact that often a quick observation of the
system may influence customers' decision about the utility of joining a system or of using a call-back option. In particular, the last
realized event often has an important impact on their decisions. A too high importance
given to the last realized event may, bias customers' decisions by inducing
illusory correlation. For example, if the last event was a service, a customer might thinks that this system has a high speed of service. If the last event was a successful retrial, a customer might thinks that it is not a bad luck to join the orbit queue, i.e., to use the call-back option. If the last event was an arrival of a primary (i.e., a newly arriving customer), the customer might thinks that it is not a good option to join and leave her contact details, i.e., to use the call-back option. 

In modern call centers, the call-back option \cite{call1,call2} allows to essentially decrease the loss probability of calls, to avoid frustration
of customers, to make more smooth load of operators and increase the effectiveness of their
work. With this option, customers who call during the epoch when all operators are busy, does not wait for service in line, but they register in the system and an operator will contact him/her for service later on.
In call centers that do not provide this option some part of the customers' service time is spent on listening customers' complaints about the long waiting time. So, using of the call-back option, we can reduce the average customer's service and
waiting time; see \cite{dud1,phung}. 

In this context, consider a call center, in which \textit{potential} customers are informed about the last realized event. If they receive a busy signal along with the information of the type of the customer in service (see below for more details), they analogously adapt their arrival rates, and either leave their contact details, so they are waiting to get called back latter, or abandon the system. So, the \textit{event-dependency} seems to be an indicator of a strategic behaviour of the customers.

Clearly, the last realized event is not a rational indicator of the quality of a service
system. However, in \cite{gencer}, based on laboratory experiments, the authors showed that the evolution of the queue size due to service completions or arrivals impacts strongly the decision of new arriving customers. 

The main contributions of the paper are summarized as follows.
\begin{itemize}
\item On the modelling side, we introduce the concept of \textit{event-dependent} arrivals in the retrial setting. In particular, we employ a multi-level event dependency, where the arrival rates depend both on the last realized event (i.e., arrival or a departure), and on the type of the last arrival (i.e., whether it is a primary or a retrial customer). We incorporate a behavioural aspect of arriving customers based on the last realized event.
\item On the technical side, we investigate the stationary behaviour at service completion epoch, as well as at an arbitrary epoch, and provide explicit expressions for various performance metrics. The effect of event-dependency on system's performance, is extensively investigated through numerical experiments. Our results indicate that event-dependency is a result of customer's strategic behaviour. Constraint optimization problems are solved and provide insights on how event-dependency affects the admission control problem. Moreover, we provide necessary and sufficient the stability conditions. The asymptotic behaviour of our system under high rate of retrials is also investigated.
\end{itemize}
\subsection{Literature review}
Queueing systems with orbits, formally known as retrial queues enables an accurate representation and quantification of real service systems with finite
waiting spaces. In such systems, customers who find upon arrival a busy server join a so-called retrial orbit and retry to access the server later. For a detailed treatment on retrial queues see the books in \cite{falin,arta}, and references therein; see also \cite{phung}.

The vast majority of works on retrial queueing systems assume the so-called classical retrial policy, under which each retrying customer conducts attempts for service independently of other customers, after exponentially distributed inter-retrial times; see e.g., \cite{dim1}. However, in specific service settings, the
intervals between the successive retrials are independent of the number
of attempting customers. In such cases, it is assumed that only the customer that is at the ``head" of
the orbit is allowed to conduct retrials, e.g., \cite{dim2}, or equivalently, the server searches for customers from the orbit after a service completion (i.e., the concept of call-back option mentioned above). Such a policy is called the constant retrial policy, and was introduced in a fully Markovian framework in \cite{fayo}. In \cite{choi}, the authors considered generally distributed retrial times. In \cite{gomez}, the author  presented an exhaustive analysis of the single-server
retrial queue with the constant retrial policy and general service and seeking times, which includes the
stability condition, the probability generating function (pgf) of the steady-state distribution
and the Laplace–Stieltjes transforms of the waiting time, busy period and idle period
distributions. Since the work in \cite{gomez}, many works have appeared that extend his analysis to systems with additional characteristics. We mention the work in \cite{baron}, where they considered the state-dependent version of the model in \cite{gomez}, using a probabilistic and effective computational method.

However, there is no generalization of the analysis for the \textit{event-dependent} version of the model in \cite{gomez}. A main objective of our work is to fill this gap. Our work differs from \cite{baron}, since the arrival-dependency is based on the last realized event instead of the observed number of orbiting customers. Moreover, the presence of primary and orbiting customers arise a multi-level event-dependent framework.

The performance analysis of queues with event-dependencies was recently introduced in \cite{legros2018}. In \cite{legros2018}, the author studied the standard (no retrials) M/G/1 queue with event-dependent arrivals. In \cite{legros2018q}, the authors studied queueing models where arrivals depend on the remaining service time. Recently, in \cite{legros2021}, the authors investigated the admission control problem with state-dependent arrivals, and provided an algorithm for dimensioning the system. In \cite{legros2022}, the author studied a G/M/1 queue with event-dependent service rate and proved that the last realized event can be efficiently used to lower the wait.

Queues with workload-dependent arrival and/or service rates have been extensively studied in the literature; e.g., \cite{beker,box,kern} (not exhaustive list).  Other
single-server queueing models have been proposed where the arrival or the service
rates depend on the waiting time of the customer in service, or in the queue \cite{box2,dau}. To our best knowledge, the work in \cite{baron} is the only that deals with the analysis of state-dependent M/G/1 retrial queue. We aim to investigate a queue with another feature of the customers' behaviour, namely \textit{multi-level event-dependency} in the retrial setting.  

The rest of the paper is summarized as follows. In Section \ref{model}, we describe the mathematical model in detail. The stability condition and the stationary analysis at service completion epochs is performed in Section \ref{mc}. The stationary analysis at an arbitrary epoch is presented in Section \ref{arb}. Explicit expressions for various performance metrics along with an asymptotic result is also presented. Extensive numerical results that reveal the effect of event-dependency both on the system's performance and on the admission control problem are presented in Section \ref{num}.
\section{Model description}\label{model}
We consider a single-server queueing system with no waiting space. The service times are iid random variables with cumulative distribution function (cdf) $B(\cdot)$, density $b(\cdot)$, Laplace-Stieltjes (LST) $\beta^{*}(.)$ and firsts moments $\bar{b}^{(k)}=(-1)^{k}\beta^{*}(0)$, $k=1,2$. The customers that find the server busy upon arrival, they abandon the system but leave their contact details; hence, we can think that they join an infinite capacity orbit queue, or equivalently, they leave their contact details so they are called back by the server in a later instant. After finishing service, a customer leaves the system and the server declares that a service has completed, and seeks for a customer from the orbit. The seeking/retrieving times are iid random variables with cdf $A(.)$, density $a(.)$ and LST $\alpha^{*}(.)$. However, a new/primary  customer may arrive during the seeking process, and in such a case, the server interrupts the seeking process, and starts serving the newly arriving customer. We assume that the interarrival, service and seeking times are mutually independent. 

We consider a multi-level event-dependent arrivals. Recall that after a service completion, there is a competition between external arrivals and retrials. The type of customer that will occupy the server influences the arrival rates for the next customers. More precisely, based on the last realized event, the next customer arrives according to a Poisson process, as follows (see also Table \ref{t1}):
\begin{itemize}
\item[-] If the last realized event is a service completion, the next primary customer will arrive at a rate $\lambda^{-}$.
\item[-] In case a primary customer has occupied the idle server, then, the first primary customer that arrive during the busy period initiated by that primary customer will arrive at a rate $\lambda^{e}$. Moreover, the subsequent primary customers (i.e., the second, third, etc arriving customers during the busy period initiated by that primary customer) will arrive at a rate $\lambda_{+}^{e}$.
\item[-] In case a retrial customer has occupied the idle server, then, the first primary customer that arrive during the busy period initiated by that retrial customer will arrive at a rate $\lambda^{r}$. Moreover, the subsequent primary customers (i.e., the second, third, etc arriving customers during the busy period initiated by that retrial customer) will arrive at a rate $\lambda_{+}^{r}$.
\end{itemize}
\begin{table}
\caption{Summary of event description and the corresponding arrival rates.}
\label{t1}
\centering
\begin{tabular}{|l|c|}\hline
\textbf{\small{Description of last event}}& \textbf{\small{Next arrival at rate}}\\
\hline \hline
Service completion & $\lambda^{-}$\\\hline
An external arrival has occupied the idle server& $\lambda^{e}$\\\hline
At least one external customer has arrived after&$\lambda_{+}^{e}$\\
the occupation of the idle server by an external customer&\\\hline
A retrial customer has occupied the idle server& $\lambda^{r}$\\\hline
At least one external customer has arrived after&$\lambda_{+}^{r}$\\
the occupation of the idle server by a retrial customer&\\\hline
\end{tabular}
\end{table}
\begin{remark}
From the customer's perspective one might expects $\lambda^{-}>\lambda^{e}\geq \lambda_{+}^{e}$, and $ \lambda^{r}\geq \lambda_{+}^{r}$. This is expected since, if a customer knows that the last realised event is an arrival that has occupied the idle server, she knows that if she decides to join the system, she will be routed to the orbit queue. So she has to wait to be called back by the server in a latter instant (i.e., $\lambda^{-}>\lambda^{e}$). Normally, the subsequent arrivals that already know that other customers have already arrived previously, they might be even more doubted to join the orbit queue (i.e., it is expected that $\lambda^{e}\geq \lambda_{+}^{e}$). Similar arguments may hold for the case $\lambda^{r}\geq \lambda_{+}^{r}$.  
\end{remark}

\section{The embedded Markov chain at service completion epochs}\label{mc}
Let $\tau_{i}$ be the time of the $i$-th departure and $X_{i}=X(\tau_{i}^{+})$ be the number of customers left in orbit just after the departure of the $i-th$ customer. Then, we can write
\begin{equation}
X_{i}=X_{i-1}-B_{i}+A_{i}(B_{i}),
\end{equation}
where $B_i\in\{0,1\}$ is the number of orbiting customers, which enter service at time the $i$th service starts (i.e. $B_i = 1$ if the $i$th cust is an orbiting customer and $B_i = 0$ if the $i$th
cust is an external customer), and
$A_{i}(B_{i})$ is the number of external arrivals during the time the $i$-th served customer stays in the service station ($A_{i}(0)$ (resp. $A_{i}(1)$) is the number of arriving customers during the service of a primary (resp. a retrial) customer). The random variable $B_i$ depends on the history of the system before the time $\tau_{i-1}$ only
through the variable $X_{i-1}$ and its conditional distribution is given by
\begin{displaymath}
\begin{array}{rl}
\mathbb{P}(B_{i}=0|X_{i-1}=n)=&(1-\delta_{0,n})\alpha^{*}(\lambda^{-}),\\
\mathbb{P}(B_{i}=1|X_{i-1}=n)=&1-\delta_{0,n}\alpha^{*}(\lambda^{-}).
\end{array}
\end{displaymath}
The service time of the $i$-th customer is independent of previous service times and the number of orbiting customers. Denote by $S$ the  corresponding service time. We now focus on the distribution of $A_{i}$. Note that since we consider event-dependent arrival rates, we must take into account all the possible events mentioned at the end of the previous section.  More precisely,\\
\textbf{Case 1:} If $X_{i}>0$, the last event is a service completion that leaves the server idle. Thus, the next customer that occupies the server is either an external customer (at a Poisson rate $\lambda^{-}$), or a registered (i.e., a retrial) customer. Therefore, the last event for the first customer who arrives during the service of $(i+1)$-th customer is either an arrival or a successful retrial/retrieval. In case the server was occupied by an external customer, the first customer will arrive at rate $\lambda^{e}$, and all subsequent customers at a rate $\lambda^{e}_{+}$. In case the server was occupied by a retrial customer, the first customer to the busy server arrives at rate $\lambda^{r}$, and all subsequent customers at a rate $\lambda^{r}_{+}$. With such a framework, the next arrival depends both on the last event (i.e., arrival or service completion), and on the type of the customer that have occupied the server in the last (arrival) event. \\
\textbf{Case 2:} If $X_{i}=0$, the last event is a service completion that leaves the system empty. The next customer that occupies the server is an external customer (at a Poisson rate $\lambda^{-}$). Therefore, the last event for the first customer that will arrive after the server's occupation is an arrival, thus will arrive at rate $\lambda^{e}$, and all the subsequent customers will arrive with rate $\lambda_{+}^{e}$.

Thus, due to the event-dependency we need to obtain the distribution of the number of arrivals in a service of length $t$ given the type of the customer that occupied the server. Denote by $N(t)$ the number of arriving customers during a service of length $t$ and let 
\begin{displaymath}
\begin{array}{rl}
\mathbb{P}_{e}(N(t)=n)=&\mathbb{P}(N(t)=n|\text{the server is occupied by a primary customer}),\\
\mathbb{P}_{r}(N(t)=n)=&\mathbb{P}(N(t)=n|\text{the server is occupied by a retrial customer}).
\end{array}
\end{displaymath}
Then, the distribution of $N(t)$, is given by the following set of differential equations:
\begin{equation}
\begin{array}{rl}
\mathbb{P}_{e}(N(0)=0)=&1,\\
\frac{d}{dt}\mathbb{P}_{pr}(N(t)=0)=&-\lambda^{e}\mathbb{P}_{e}(N(t)=0),\\
\frac{d}{dt}\mathbb{P}_{e}(N(t)=1)=&-\lambda_{e}^{+}\mathbb{P}_{e}(N(t)=1)+\lambda^{e}\mathbb{P}_{e}(N(t)=0),\\
\frac{d}{dt}\mathbb{P}_{e}(N(t)=n)=&-\lambda_{e}^{+}\mathbb{P}_{e}(N(t)=n)+\lambda_{e}^{+}\mathbb{P}_{e}(N(t)=n-1),\,n\geq 2.
\end{array}
\label{e1}
\end{equation}
Similarly,
\begin{equation}
\begin{array}{rl}
\mathbb{P}_{r}(N(0)=0)=&1,\\
\frac{d}{dt}\mathbb{P}_{r}(N(t)=0)=&-\lambda^{r}\mathbb{P}_{r}(N(t)=0),\\
\frac{d}{dt}\mathbb{P}_{r}(N(t)=1)=&-\lambda_{r}^{+}\mathbb{P}_{r}(N(t)=1)+\lambda^{r}\mathbb{P}_{r}(N(t)=0),\\
\frac{d}{dt}\mathbb{P}_{r}(N(t)=n)=&-\lambda_{r}^{+}\mathbb{P}_{r}(N(t)=n)+\lambda_{r}^{+}\mathbb{P}_{r}(N(t)=n-1),\,n\geq 2.
\end{array}
\label{e2}
\end{equation}
The solutions of systems \eqref{e1}, \eqref{e2} are respectively
\begin{equation}
\begin{array}{rl}
\mathbb{P}_{e}(N(t)=0)=&e^{-\lambda^{e}}t,\\
\mathbb{P}_{e}(N(t)=n)=&\frac{\lambda^{e}}{\lambda_{e}^{+}-\lambda^{e}}\left(\frac{\lambda_{e}^{+}}{\lambda_{e}^{+}-\lambda^{e}}\right)^{n-1}[e^{-\lambda^{e}}t-e^{-\lambda_{e}^{+}}t\sum_{k=0}^{n-1}\frac{((\lambda_{e}^{+}-\lambda^{e})t)^{k}}{k!}],\,n\geq 1.
\end{array}\label{sol1}
\end{equation}
and
\begin{equation}
\begin{array}{rl}
\mathbb{P}_{r}(N(t)=0)=&e^{-\lambda^{r}}t,\\
\mathbb{P}_{r}(N(t)=n)=&\frac{\lambda^{r}}{\lambda_{r}^{+}-\lambda^{r}}\left(\frac{\lambda_{r}^{+}}{\lambda_{r}^{+}-\lambda^{r}}\right)^{n-1}[e^{-\lambda^{r}}t-e^{-\lambda_{r}^{+}}t\sum_{k=0}^{n-1}\frac{((\lambda_{r}^{+}-\lambda^{r})t)^{k}}{k!}],\,n\geq 1.
\end{array}\label{sol2}
\end{equation}

Note that the arrival processes are modified Poisson processes where the first interarrival time follows a different distribution than the other interarrival times, and at the same time depend also on the type of the customer that occupy the server (i.e., a primary or a retrial customer). Then, for $i,n\geq 0$,
\begin{displaymath}
\begin{array}{c}
\mathbb{P}(A_{i}(0)=n|X_{i}\geq 0)=\int_{0}^{\infty}\mathbb{P}_{e}(N(t)=n)b(t)dt=b_{n}^{e}.
\end{array}
\end{displaymath} 
Similarly, for $i,n\geq 0$
\begin{displaymath}\begin{array}{c}
\mathbb{P}(A_{i}(1)=n|X_{i}>0)=\int_{0}^{\infty}\mathbb{P}_{r}(N(t)=n)b(t)dt=b_{n}^{r}.
\end{array}
\end{displaymath} 

Let $A_{k}(z)=\sum_{n=0}^{\infty}b_{n}^{k}z^{n}=\sum_{n=0}^{\infty}\int_{0}^{\infty}\mathbb{P}_{k}(N(t)=n)z^{n}b(t)dt$, $|z|\leq 1$, $k=e,\ r$, i.e., $A_{k}(z)$ is the probability generating function of the number of customers that arrive at the system during the service time of a customer of type $k$, $k=e,r$. Then, extensive computations leads to
\begin{displaymath}
\begin{array}{l}
A_{k}(z)=\frac{\beta^{*}(\lambda^{k})(\lambda_{+}^{k}-\lambda^{k})(z-1)-\lambda^{k}z\beta^{*}(\lambda_{+}^{k}(1-z))}{\lambda_{+}^{k}(1-z)-\lambda^{k}},\,k=e,r.
\end{array}
\end{displaymath}
\begin{remark}
Note that if we assume \textit{no event-dependency}, i.e., for $k=e,r,$ $\lambda_{+}^{k}=\lambda^{k}=\lambda$, then, $A_{k}(z)=\beta^{*}(\lambda(1-z))$.
\end{remark}

The one-step transition probabilities $p_{m,n}= \mathbb{P}(X_{i} = n|X_{i-1} = m)$ are given by the formulae:
\begin{displaymath}
\begin{array}{rl}
p_{m,n}=&(1-\alpha^{*}(\lambda^{-})b_{n-m}^{e}+\alpha^{*}(\lambda^{-})b_{n-m+1}^{r},\,m=1,2,\ldots,n,\\
p_{0,n}=&b^{e}_{n},\,n\geq 0,\\
p_{m+1,m}=&\alpha^{*}(\lambda^{-})b_{0}^{r},\,m\geq 0.
\end{array}
\end{displaymath}
\begin{theorem}\label{erg}
 Let $X_{i}$ be the orbit length at the time of the $i$th departure, $i\geq 1$. Then, $\{X_{i},
i\geq 1\}$ is ergodic if and only if 
\begin{equation}
\begin{array}{l}
\bar{b}<\frac{\alpha^{*}(\lambda^{-})[\lambda_{+}^{r}+(\lambda^{r}-\lambda_{+}^{r})\beta^{*}(\lambda^{r})]}{\lambda^{r}[\lambda_{+}^{e}+(\lambda_{+}^{r}-\lambda_{+}^{e})\alpha^{*}(\lambda^{-})]}-\frac{(1-\alpha^{*}(\lambda^{-}))(\lambda^{e}-\lambda_{+}^{e})\beta^{*}(\lambda^{e})}{\lambda^{e}[\lambda_{+}^{e}+(\lambda_{+}^{r}-\lambda_{+}^{e})\alpha^{*}(\lambda^{-})]}
\end{array}\label{ergo}
\end{equation}
\end{theorem}
\begin{proof}
See \ref{ap1}.
\end{proof}
Let $\pi_{n}$, $n\geq 0$, be the stationary probability of being in state $n$. Then, the Kolmogorov equations reads: 
\begin{equation}
\begin{array}{rl}
\pi_{n}=&\pi_{0}b_{n}^{e}+(1-\delta_{0n})(1-\alpha^{*}(\lambda^{-}))\sum_{j=1}^{n}\pi_{j}b_{n-j}^{e}\vspace{2mm}\\&+\alpha^{*}(\lambda^{-})\sum_{j=1}^{n+1}\pi_{j}b_{n+1-j}^{r},\,n\geq 0.
\end{array}\label{bal}
\end{equation}
Let $\Pi(z)=\sum_{n=0}^{\infty}\pi_{n}z^{n}$, $|z|\leq 1$. In Theorem \ref{t2}, we obtain $\Pi(z)$, $\pi_{0}$, and give the condition of existence of $\Pi(z)$, which also ensures the stationary regime.
\begin{theorem}\label{t2}
Under the stability condition \eqref{ergo}, we have
\begin{displaymath}
\Pi(z)=\pi_{0}\alpha^{*}(\lambda^{-})\frac{zA_{e}(z)-A_{r}(z)}{\alpha^{*}(\lambda^{-})(zA_{e}(z)-A_{r}(z))+z(1-A_{e}(z))},
\end{displaymath}
where
\begin{displaymath}
\begin{array}{l}
\pi_{0}=\frac{\lambda^{e}\alpha^{*}(\lambda^{-})[\lambda_{+}^{r}(1-\lambda^{r}\bar{b})+\beta^{*}(\lambda^{r})(\lambda^{r}-\lambda_{+}^{r})]-\lambda^{r}(1-\alpha^{*}(\lambda^{-}))[\lambda_{+}^{e}\lambda^{e}\bar{b}+(1-\beta^{*}(\lambda^{e})(\lambda^{e}-\lambda_{+}^{e})]}{\alpha^{*}(\lambda^{-})\left[\lambda^{e}[\lambda_{+}^{r}(1-\lambda^{r}\bar{b})+\beta^{*}(\lambda^{r})(\lambda^{r}-\lambda_{+}^{r})]+\lambda^{r}[\lambda_{+}^{e}\lambda^{e}\bar{b}+(1-\beta^{*}(\lambda^{e})(\lambda^{e}-\lambda_{+}^{e})]\right]}.
\end{array}
\end{displaymath}
Asking $\pi_{0}>0$ we have that \eqref{ergo} is
also necessary for the ergodicity of the chain.
\end{theorem}
\begin{proof}
The proof is straightforward by using \eqref{bal} and applying the generating function approach. The normalization condition implies the expression for $\pi_{0}$.
\end{proof}
\begin{remark}
Note that our model exhibits a behaviour closely related to the stochastic decomposition behaviour, which normally arise in the standard (i.e., no even-dependency) retrial systems. In particular, when $\lambda^{r}=\lambda^{-}$, $\lambda_{+}^{r}=\lambda^{+}$,
\begin{displaymath}
\Pi(z)=\Pi_{M/G/1}^{(event)}(z)\chi_{1}(z)\chi_{2}(z),
\end{displaymath} 
where $\Pi_{M/G/1}^{(event)}(z)$ is the pgf of the number of customers at service completion epochs in the standard M/G/1 queue with event-dependent arrivals \cite{legros2018}, and
\begin{displaymath}
\begin{array}{rl}
\chi_{1}(z)=&\frac{z-A_{r}(z)}{\alpha^{*}(\lambda^{-})(zA_{e}(z)-A_{r}(z))+z(1-A_{e}(z))}\times\frac{\alpha^{*}(\lambda^{-})(1+A_{e}^{(1)}(1)-A_{r}^{(1)})(1))}{1-A_{r}^{(1)}(1)},\\
\chi_{2}(z)=&\frac{zA_{e}(z)-A_{r}(z)}{zB(z)-A_{r}(z)}\times\frac{(1+\lambda^{+}\bar{b}-A_{r}^{(1)})(1))}{1+A_{e}^{(1)}(1)-A_{r}^{(1)}(1)},
\end{array}
\end{displaymath}
where $A_{k}^{(1)}(1)$, $k=e,r,$ are given in Corollary \ref{cor1}. Moreover, $\chi_{1}(z)$ is the pgf of the number of orbiting customers given the system is idle. However, although $\chi_{2}(1)=1$, $\chi_{2}(z)$, is not obvious that constitutes a pgf.
\end{remark}
\section{Performance analysis at arbitrary instants}\label{arb}
Let $X(t)$ be the number of orbiting customers, $C(t)$ the state of the server, and $I(t)$ the last realized event at time $t$, with values as described in Table \ref{tab2}.
\begin{table}
\caption{Description of the states of $I(t)$.}
\label{tab2}
\centering
\begin{tabular}{|l|l|c|}\hline
\textbf{Symbol}&\textbf{\small{Last realized event}}\\
\hline \hline
$E_{1}$&Service completion \\\hline
$E_{2}$& an external arrival occupied the server \\ \hline
$E_{3}$&a retrial customer occupied the server \\ \hline
$E_{4}$& the 1st external customer during\\& the busy period initiated in $E_{2}$, has arrived.\\\hline
$E_{5}$& at least one customer has arrived after $E_{4}$\\\hline
$E_{6}$& the 1st external customer during\\& the busy period initiated in $E_{3}$, has arrived.\\\hline
$E_{7}$& at least one customer has arrived after $E_{6}$\\\hline
\end{tabular}
\end{table}
Let $Z(t)$ the remaining time until the next service (when $C(t)=1$), or seeking completion (when $C(t)=0$) at time $t$. Then $\{(C(t),X(t),I(t),Z(t));t\geq 0\}$ is an irreducible continuous Markov chain with state space $\{(0,0,E_{1})\}\cup\{(0,j,E_{1},r):j\geq 1,r\geq 0\}\cup\{(1,j,E_{k},r):j\geq 0,r\geq 0,k=2,3\}\cup\{(1,j,E_{k},r):j\geq 1,r\geq 0,k=4,6\}\cup\{(1,j,E_{k},r):j\geq 2,r\geq 0,k=5,7\}$.

Let also
\begin{displaymath}
\begin{array}{rl}
p_{0,0}(t)=&P(C(t)=0,X(t)=0,I(t)=E_{1}),\vspace{2mm}\\
p_{0,j}(r,t)=&P(C(t)=0,X(t)=j,I(t)=E_{1},Z(t)\in (r,r+dr]),\,j\geq 1,\vspace{2mm}\\
p^{(k)}_{1,j}(r,t)=&P(C(t)=1,X(t)=j,I(t)=E_{k},Z(t)\in (r,r+dr]),\,j\geq 0,k=2,3,\vspace{2mm}\\
p^{(k)}_{1,j}(r,t)=&P(C(t)=1,X(t)=j,I(t)=E_{k},Z(t)\in (r,r+dr]),\,j\geq 1,k=4,6,\vspace{2mm}\\
p^{(k)}_{1,j}(r,t)=&P(C(t)=1,X(t)=j,I(t)=E_{K},Z(t)\in (r,r+dr]),\,j\geq 2,k=5,7.
\end{array}
\end{displaymath}
We are interesting in the steady-state counterparts (as $t\to\infty$) of these probabilities. 
\begin{lemma}\label{th3}
Let $p_{0,0}=\lim_{t\to\infty}p_{0,0}(t)$, $p_{0,j}(r)=\lim_{t\to\infty}p_{0,j}(r,t)$, and $p_{1,j}^{(k)}(r)=\lim_{t\to\infty}p_{1,j}^{(k)}(r,t)$, $k=2,\ldots,7$. Then:
\begin{equation}
\begin{array}{rll}
\lambda^{-}p_{0,0}=&\sum_{k=2}^{3}p_{1,0}^{(k)}(0),&\\
-\frac{d}{dr}p_{0,j}(r)=&-\lambda^{-}p_{0,j}(r)+a(r)\sum_{k=2}^{7}p_{1,j}^{(k)}(0),&j\geq 1,\\
-\frac{d}{dr}p_{1,j}^{(2)}(r)=&-\lambda^{e}p_{1,j}^{(2)}(r)+\lambda^{-}p_{0,j}b(r),&j\geq 0,\\
-\frac{d}{dr}p_{1,j}^{(3)}(r)=&-\lambda^{r}p_{1,j}^{(3)}(r)+p_{0,j+1}(0)b(r),&j\geq 0,\\
-\frac{d}{dr}p_{1,j}^{(4)}(r)=&-\lambda_{+}^{e}p_{1,j}^{(4)}(r)+\lambda^{e}p_{1,j}^{(2)}(r),&j\geq 1,\\
-\frac{d}{dr}p_{1,j}^{(5)}(r)=&-\lambda_{+}^{e}p_{1,j}^{(5)}(r)+\lambda_{+}^{e}(p_{1,j-1}^{(5)}(r)+p_{1,j-1}^{(4)}(r)),&j\geq 2,\\
-\frac{d}{dr}p_{1,j}^{(6)}(r)=&-\lambda_{+}^{r}p_{1,j}^{(6)}(r)+\lambda^{r}p_{1,j}^{(3)}(r),&j\geq 1,\\
-\frac{d}{dr}p_{1,j}^{(7)}(r)=&-\lambda_{+}^{r}p_{1,j}^{(7)}(r)+\lambda_{+}^{r}(p_{1,j-1}^{(7)}(r)+p_{1,j-1}^{(6)}(r)),&j\geq 2.
\end{array}\label{bnm}
\end{equation}
\end{lemma}
\begin{proof}
See \ref{ap2}.
\end{proof}
Let for $Re(s)\geq 0$, $|x|\leq 1$,
\begin{displaymath}
\begin{array}{rl}
P_{0}^{*}(s,z)=&\sum_{j=1}^{\infty}\int_{0}^{\infty}e^{-sr}p_{0,j}(r)drz^{j},\\
P_{1,k}^{*}(s,z)=&\sum_{j=0}^{\infty}\int_{0}^{\infty}e^{-sr}p_{1,j}^{(k)}(r)drz^{j},\,k=2,3,\\
P_{1,k}^{*}(s,z)=&\sum_{j=1}^{\infty}\int_{0}^{\infty}e^{-sr}p_{1,j}^{(k)}(r)drz^{j},\,k=4,6,\\
P_{1,k}^{*}(s,z)=&\sum_{j=2}^{\infty}\int_{0}^{\infty}e^{-sr}p_{1,j}^{(k)}(r)drz^{j},\,k=5,7.
\end{array}
\end{displaymath}

\begin{theorem}\label{basic}
The stationary distribution of $(C,X,I)$ has the following pgfs
\begin{equation}
\begin{array}{rl}
P_{0}^{*}(0,z)=&\frac{p_{0,0}z(1-\alpha^{*}(\lambda^{-}))(A_{e}(z)-1)}{\alpha^{*}(\lambda^{-})(zA_{e}(z)-A_{r}^{*}(z))+z(1-A_{e}(z))},\vspace{2mm}\\
P_{1,2}^{*}(0,z)=&\lambda^{-}\frac{1-\beta^{*}(\lambda^{e})}{\lambda^{e}}(p_{0,0}+P_{0}^{*}(0,z)),\vspace{2mm}\\
P_{1,3}^{*}(0,z)=&\frac{1-\beta^{*}(\lambda^{r})}{\lambda^{r}}K(z),\vspace{2mm}\\
\sum_{k=4}^{5}P_{1,k}^{*}(0,z)=&\frac{\lambda^{-}z[1-\beta^{*}(\lambda^{e})-\frac{\lambda^{e}(1-\beta^{*}(\lambda_{+}^{e}(1-z)))}{\lambda_{+}^{e}(1-z)}]}{\lambda_{+}^{e}(1-z)-\lambda^{e}}(p_{0,0}+P_{0}^{*}(0,z)),\vspace{2mm}\\
\sum_{k=6}^{7}P_{1,k}^{*}(0,z)=&\frac{z[1-\beta^{*}(\lambda^{r})-\frac{\lambda^{r}(1-\beta^{*}(\lambda_{+}^{r}(1-z)))}{\lambda_{+}^{r}(1-z)}]}{\lambda_{+}^{r}(1-z)-\lambda^{r}}K(z),
\end{array}\label{t1s1}
\end{equation}
where
\begin{displaymath}
\begin{array}{rl}
K(z)=&\frac{\lambda^{-}\alpha^{*}(\lambda^{-})[p_{0,0}(A_{e}(z)-1)+A_{e}(z)P_{0}^{*}(0,x)]}{z-\alpha^{*}(\lambda^{-})A_{r}(z)},\vspace{2mm}\\
p_{0,0}=&\frac{\alpha^{*}(\lambda^{-})(\lambda^{e}t_{r}+\lambda^{r}t_{e})-\lambda^{r}t_{e}}{\alpha^{*}(\lambda^{-})[(1+\lambda^{-}\bar{b})\lambda^{e}t_{r}+\lambda^{-}\bar{b}\lambda^{r}t_{e}]},%
\end{array}
\end{displaymath}where 
\begin{displaymath}
\begin{array}{rl}
t_{e}=&(\lambda^{e}-\lambda_{+}^{e})(1-\beta^{*}(\lambda^{e}))+\lambda^{e}\lambda_{+}^{e}\bar{b},\\
t_{r}=&(\lambda^{r}-\lambda_{+}^{r})\beta^{*}(\lambda^{r})+\lambda_{+}^{r}(1-\lambda^{r}\bar{b}).
\end{array}
\end{displaymath}
\end{theorem}
\begin{proof}
See \ref{ap3}.
\end{proof}
Note that asking $p_{0,0}>0$, we obtain the necessary stability condition, which is the same as the one in \eqref{ergo}. 
\subsection{The event-independent case}
We now consider the event-independent case, where $\lambda^{k}=\lambda_{+}^{k}=\lambda$, $k=e,r$. Then, using the previous results, our model reduces to the one in \cite{gomez}, where the event-independent case was treated. Indeed, it easy to see that when $\lambda^{k}=\lambda_{+}^{k}=\lambda$, $k=e,r$, then $A_{k}(z)=\beta^{*}(\lambda-\lambda z)$, and $\pi_{0}=p_{0,0}=1-\frac{\lambda\bar{b}}{\alpha^{*}(\lambda)}$ with $\lambda\bar{b}<\alpha^{*}(\lambda)$ being the stability condition; see Theorems 1, 2 in \cite{gomez}.
\subsection{Performance metrics}
Having obtained explicitly the pgfs, we can have in closed for the basic performance metrics.
\begin{corollary}\label{c1}
The probabilities of server's state are:
\begin{equation*}
\begin{array}{rl}
P(C=0)=&p_{0,0}+P_{0}^{*}(0,1)=\frac{\lambda^{e}t_{r}}{(1+\lambda^{-}\bar{b})\lambda^{e}t_{r}+\lambda^{-}\bar{b}\lambda^{r}t_{e}},\\
P(C=1,I=E_{2})=&P_{1,2}^{*}(0,1)=\frac{\lambda^{-}t_{r}(1-\beta^{*}(\lambda^{e}))}{(1+\lambda^{-}\bar{b})\lambda^{e}t_{r}+\lambda^{-}\bar{b}\lambda^{r}t_{e}},\\
P(C=1,I=E_{3})=&P_{1,3}^{*}(0,1)=\frac{\lambda^{-}t_{e}(1-\beta^{*}(\lambda^{r}))}{(1+\lambda^{-}\bar{b})\lambda^{e}t_{r}+\lambda^{-}\bar{b}\lambda^{e}t_{r}},\\
P(C=1,I=E_{4})+P(C=1,I=E_{5})=&P_{1,4}^{*}(0,1)+P_{1,5}^{*}(0,1)\\=&\frac{\lambda^{-}\lambda^{r}t_{r}}{(1+\lambda^{-}\bar{b})\lambda^{e}t_{r}+\lambda^{-}\bar{b}\lambda^{r}t_{e}}(\bar{b}-\frac{1-\beta^{*}(\lambda^{e})}{\lambda^{e}}),\\
P(C=1,I=E_{6})+P(C=1,I=E_{7})=&P_{1,6}^{*}(0,1)+P_{1,7}^{*}(0,1)\\=&\frac{\lambda^{-}\lambda^{r}t_{e}}{(1+\lambda^{-}\bar{b})\lambda^{e}t_{r}+\lambda^{-}\bar{b}\lambda^{r}t_{e}}(\bar{b}-\frac{1-\beta^{*}(\lambda^{r})}{\lambda^{r}}).
\end{array}\label{t11}
\end{equation*}

\end{corollary}
\begin{proof}
From the results obtained in Theorem \ref{basic} and using the normalization condition the Corollary \ref{c1} is proved after heavy but straightforward computations.
\end{proof}
The following corollary provides the $TH_{S}$, the expected orbit queue length $E(X)$, and the expected sojourn time $E(S)$.
\begin{corollary}\label{cor1}
We have
\begin{equation}
\begin{array}{rl}
E(X)=&p_{0,0}[(1-\alpha^{*}(\lambda^{-}))(1+\lambda^{-}\bar{b})G-\frac{\lambda^{-}\alpha^{*}(\lambda^{-})}{1-\alpha^{*}(\lambda^{-})}S_{r}]\\
&+\bar{b}F+\lambda^{-}P(C=0)[S_{e}+\frac{\alpha^{*}(\lambda^{-})}{1-\alpha^{*}(\lambda^{-})}S_{r}],\vspace{2mm}\\
TH_{S}=&\frac{\lambda^{-}(\lambda^{e}t_{r}+\lambda^{r}t_{e})}{\lambda^{e}t_{r}+\lambda^{-}\bar{b}(\lambda^{e}t_{r}+\lambda^{r}t_{e})},\vspace{2mm}\\
E(S)=&\frac{E(X)}{TH_{S}},
\end{array}\label{metr}
\end{equation}
where $p_{0,0}$, $t_{r}$, $t_{e}$, $P(C=0)$ are given in Theorem \ref{basic} and Corollary \ref{c1}, and
\begin{displaymath}
\begin{array}{rl}
S_{k}:=&\frac{\lambda^{k}-\lambda_{+}^{k}}{\lambda^{k}}(\bar{b}-\frac{1-\beta^{*}(\lambda^{k})}{\lambda^{k}})+\frac{\lambda_{+}^{k}\bar{b}^{(2)}}{2},\,k=e,r,\\
G:=&\frac{\alpha^{*}(\lambda^{-})[(2A_{e}^{(1)}(1)+A_{e}^{(2)}(1))(1-A_{r}^{(1)}(1))+A_{e}^{(1)}(1)A_{r}^{(2)}(1)]}{2[A_{e}^{(1)}(1)(\alpha^{*}(\lambda^{-})-1)+\alpha^{*}(\lambda^{-})(1-A_{r}^{(1)}(1))]^{2}},\vspace{2mm}\\
F:=&\frac{[p_{0,0}(1-\alpha^{*}(\lambda^{-}))G+A_{e}^{(1)}(1)P(C=0)]-P_{0}^{*}(0,1)(1-\alpha^{*}(\lambda^{-}A_{r}^{(1)}(1)))}{(1-\alpha^{*}(\lambda^{-}))^{2}},
\end{array}
\end{displaymath}
with
\begin{displaymath}
\begin{array}{rl}
A_{k}^{(1)}(1)=&\frac{(\lambda^{k}-\lambda_{+}^{k})(1-\beta^{*}(\lambda^{k}))+\lambda^{k}\lambda_{+}^{k}\bar{b}}{\lambda^{k}},\\
A_{k}^{(2)}(1)=&\lambda_{+}^{k}(2\bar{b}+\lambda_{+}^{k}\bar{b}^{(2)})-2\frac{\lambda_{+}^{k}}{\lambda^{k}}A_{k}^{(1)}(1),
\end{array}
\end{displaymath}
for $k=e,r$.
\end{corollary}
\begin{proof}
See \ref{ap4}.
\end{proof}
\subsection{Asymptotic behaviour under high rate of retrials}
Let $P(z)=p_{0,0}+P^{*}_{0}(0,z)+x\sum_{k=2}^{7}P^{*}_{1,k}(0,z)$. Then,
\begin{displaymath}
\begin{array}{c}
\lim_{a^{*}(\lambda^{-})\to 1}P(z)=P^{(\infty)}(z),
\end{array}
\end{displaymath}
where $P^{(\infty)}(z)$ is the pgf of the number of customers in the system obtained in the seminal paper \cite{legros2018}, by assuming that $\lambda^{r}=\lambda^{-}$, and $\lambda^{e}=\lambda_{+}^{e}=\lambda_{+}^{r}=\lambda^{+}$. Note that in such a case, the customers who find the server busy repeat their
calls almost immediately. Note also that if we further assume $\lambda^{-}=\lambda_{+}=\lambda$, $P^{(\infty)}(z)$ coincides with the pgf of the number of customers in the standard M/G/1 queue.
\begin{theorem}
As $\alpha^{*}(\lambda^{-})\to 1$,
\begin{displaymath}
\begin{array}{c}
\frac{2\lambda^{r}t_{e}(1-\alpha^{*}(\lambda^{-}))}{\alpha^{*}(\lambda^{-})[(1+\lambda^{-}\bar{b})\lambda^{e}tr+\lambda^{-}\bar{b}\lambda^{r}t_{e}]}\leq\sum_{n=0}^{\infty}[P(X=n)-P^{(\infty)}(X=n)]\leq \frac{2\lambda^{r}t_{e}(1-\alpha^{*}(\lambda^{-}))}{\alpha^{*}(\lambda^{-})\lambda^{e}tr}.
\end{array}
\end{displaymath}
\end{theorem}
\begin{proof}
The proof of follows the steps given in the paper \cite{artatop}, and further details are omitted.
\end{proof}
Therefore, we have derived a measure of the proximity between the steady state distributions for the classical M/G/1 queueing system with event dependent arrivals \cite{legros2018} and our
queueing system. The importance of these bounds is to provide upper and
lower estimates for the distance between both distributions.
\section{Numerical results}\label{num}
Our aim here is two fold. First, to focus on the effect of event-dependency on system performance. In particular, to investigate how the event-dependent arrivals, i.e., the customers' behaviour based on the last realized event, affect the major performance metrics of our system, and to investigate a possible strategic behaviour; see subsection \ref{ex1}. Second, we aim to investigate the admission control problem. In particular, a controller has to determine at arrival of a new customer whether we allow him/her to enter the system or whether we will reject him/her. We set up and solve constrained optimisation problems that maximize the system generated throughput subject to certain constraints on the expected number of orbiting customers; see subsection \ref{ex2}.
\subsection{Example 1: Effect of event-dependency on system performance}\label{ex1}
Let $(\lambda^{e},\lambda_{+}^{e},\lambda^{r},\lambda_{+}^{r})=\lambda^{+} q=\lambda^{+} (q_{1},q_{2},q_{3},q_{4})$. Set $\lambda^{+}=0.3$, $\mu=2.5$ and $q=(q_{1},q_{2},q_{3},q_{4})=(0.5,0.4,0.6,0.4)$. Moreover, let $B\sim Erlang(M,\mu)$, $A\sim Erlang(N,\alpha)$.

Our aim is to investigate the effect of phases of service times and retrial times on $E(X)$ for increasing values of $\lambda^{-}$. Moreover, we aim to investigate how the relation among $\lambda^{-}$ and $\lambda^{+}$ affects $E(X)$.
\begin{figure}
\centering
\includegraphics[scale=0.45]{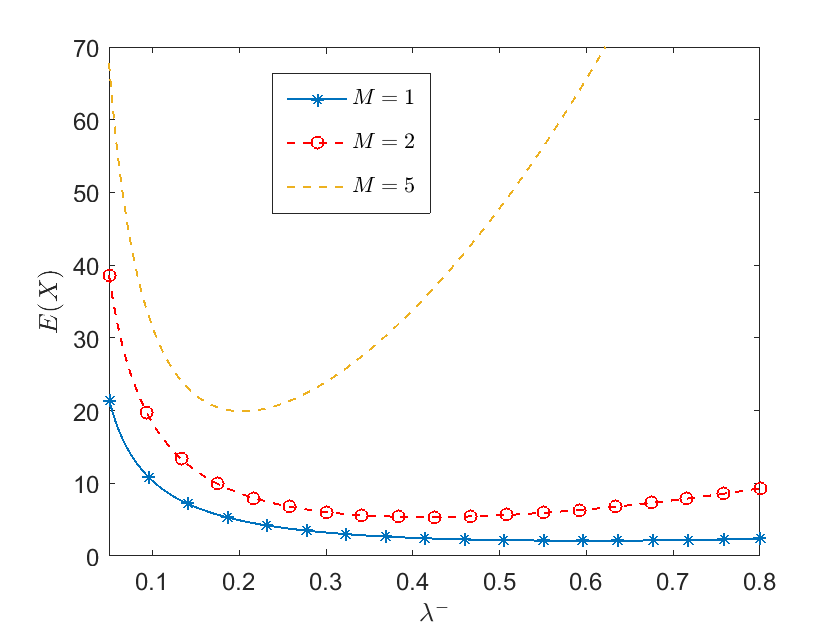}
\caption{Effect of service phases when $N=2$, $\alpha=3.5$.}
\label{st0}
\end{figure}

Note that by increasing the number of phases of service times, $E(X)$ is also increasing as expected (see Figure \ref{st0}). Moreover, as $\lambda^{-}<\lambda^{+}$, $E(X)$ decreases, and as $\lambda^{-}\geq \lambda^{+}$, $E(X)$ increases. Thus, under such a setting, by keeping the arrival rates after a service time lower than the arrival rates after an arrival, we ensure a better performance. This is due to the fact that keeping lower as possible $\lambda^{-}$ with respect to $\lambda^{+}$, we give better chances for the blocked (i.e., retrial customers) to connect with the server. Thus, under the current setting, the customers feel more comfortable to arrive after an arrival, and connect with the server as retrial customers (as long as $\lambda^{-}<\lambda^{+}$).
\begin{figure}
\centering
\includegraphics[scale=0.45]{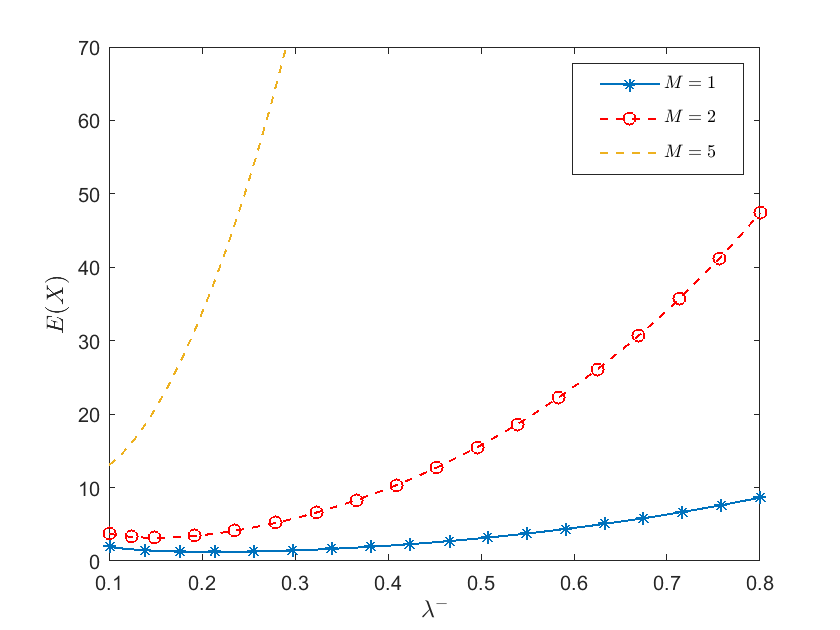}
\caption{Effect of decreasing $\alpha$ from 3.5 to 1.5 ($N=2$).}
\label{st}
\end{figure}

Figure \ref{st} indicate that by decreasing $\alpha$ from 3.5 to 1.5, we cannot have the advantage of the previous setting. Thus, even if $\lambda^{-}<\lambda^{+}$, by increasing $\lambda^{-}$, $E(X)$ increases as expected.
\begin{figure}
\centering
\includegraphics[scale=0.45]{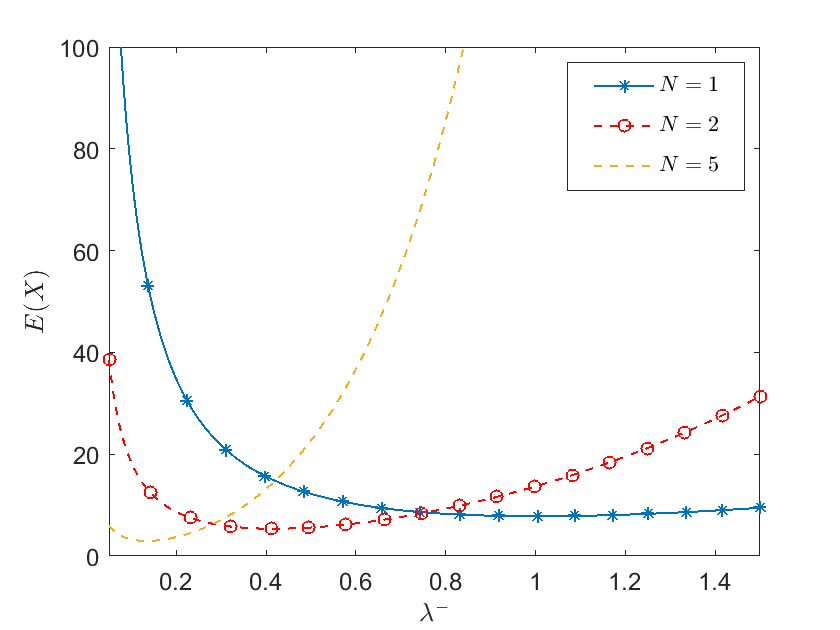}
\caption{Effect of phases of retrial times ($\alpha=3.5$, $M=2$).}
\label{st1}
\end{figure}
Furthermore, by increasing the number of phases of retrial times, we loose the advantage of the previous setting (see Figure \ref{st1}). This is because in such a case, there is a longer delay for the retrial customers to connect with the server. 
\subsection{Example 2: Admission control}\label{ex2}
Our goal is to determine the optimal joining probabilities $q=(q_{1},q_{2},q_{3},q_{4})$ that maximize the throughput ($TH_{S}$) generated by the system, subject to constraints on the maximum attained service level on $E(X)$, and the stability.

\begin{displaymath}
\left\{\begin{array}{l}
\text{Maximize }TH_{S},\vspace{2mm}\\
\text{subject to}\vspace{2mm}\\
E(X)\leq \overline{E(X)},\vspace{2mm}\\
\bar{b}<\frac{\alpha^{*}(\lambda^{-})[\lambda_{+}^{r}+(\lambda^{r}-\lambda_{+}^{r})\beta^{*}(\lambda^{r})]}{\lambda^{r}[\lambda_{+}^{e}+(\lambda_{+}^{r}-\lambda_{+}^{e})\alpha^{*}(\lambda^{-})]}-\frac{(1-\alpha^{*}(\lambda^{-}))(\lambda^{e}-\lambda_{+}^{e})\beta^{*}(\lambda^{e})}{\lambda^{e}[\lambda_{+}^{e}+(\lambda_{+}^{r}-\lambda_{+}^{e})\alpha^{*}(\lambda^{-})]},\\
q_{2}<q_{1},\,q_{4}<q_{3},\\
0\leq q_{i}\leq 1,\,i=1,2,3,4,
\end{array}\right.
\end{displaymath}
where $\bar{b}=\frac{M}{\mu}$, $\beta^{*}(s)=(\frac{\mu}{\mu+s})^{M}$, $\alpha^{*}(s)=(\frac{\alpha}{\alpha+s})^{N}$, and $(\lambda^{e},\lambda_{+}^{e},\lambda^{r},\lambda_{+}^{r})=\lambda^{+}\times (q_{1},q_{2},q_{3},q_{4})$. Note here that we added two constraints, related to the ordering if the admission probabilities $(q_{1},q_{2},q_{3},q_{4})$. In Tables \ref{opt1}, \ref{on}, \ref{on1} we assumed that $q_{2}<q_{1}$, $q_{4}<q_{3}$, having in mind that if a customer knows that another one has already arrived, it is less likely to join. In Table \ref{nont}, we excluded them, and we observed a different behaviour.

Set $\lambda^{+}=2$, $\mu=1.5$, $\alpha=3$, $M=4$, $N=3$, $\overline{E(X)}=20$. The optimal joining probabilities as functions of $\lambda^{-}$ that maximize $TH_{S}$ are given in Table \ref{opt1}. We observe that by increasing $\lambda^{-}$, $TH_{S}$ increases too. It seems that when $\lambda^{-}<\lambda^{+}$, it is better to reject newly arriving customers after the occupation of the server by a primary customer (i.e., small $q_{1}$, $q_{2}$ compared with $q_{3}$, $q_{4}$). When $\lambda^{-}=\lambda^{+}$, $TH_{S}$ is maximized by rejecting all the newly arriving customers that arrive after the occupation of the server by a retrial customer. 
\begin{table}

\centering
\begin{tabular}{|c| |c|c|}\hline
$\lambda^{-}$ & $q_{opt}=(q_{1},q_{2},q_{3},q_{4})$& Optimal $TH_{S}$\\
\hline \hline
0.1 & $(0.0001,0.0001,0.3148,0.1431)$ & $0.0788$\\
1 &$(0.0547,0.0287,0.1719,0.033)$ & $0.3082$\\
2 & $(0.0727,0.027,0,0)$ & $0.3289$\\
4&$(0.0094,0.0048,0.0962,0.0271)$&$0.3452$\\
6.1&$(0.0006,0.0003,0.2856,0.0687)$&$0.354$\\
\hline
\end{tabular}
\caption{Optimal values of joining probabilities as functions of $\lambda^{-}$.}
\label{opt1}
\end{table}

Set $\lambda^{-}=1$, $\lambda^{+}=0.5$, $\mu=1.5$, $\alpha=3$, $M=4$, $\overline{E(X)}=20$. In Table \ref{on} we observe that system performance decreases as the number of phases of retrial times increase. We observe that when $N=30$, $TH_{S}$ will be maximized if we reject customers that arrive after the arrival of a primary customer, and it is more likely to accept customers that arrive after a successful retrial that occupies the idle server.
\begin{table}
\centering
\begin{tabular}{|c| |c|c|}\hline
$N$ & $q_{opt}=(q_{1},q_{2},q_{3},q_{4})$& Optimal $TH_{S}$\\
\hline \hline
1 & $(1,0,0.7199,0)$ & $0.328$\\
2 &$(0.4614,0.1678,0.6339,0.1532)$ & $0.3221$\\
5 & $(0.2669,0.0383,0.0309,0.0157)$ & $0.2916$\\
15&$(0.0062,0.0031,0.3373,0.1446)$&$0.2737$\\
30&$(0.0001,0,0.3269,0.1533)$&$0.2727$\\
\hline
\end{tabular}
\caption{Optimal values of joining probabilities as functions of $N$.}\label{on}
\end{table}
\begin{table}
\centering
\begin{tabular}{|c| |c|c|}\hline
$M$ & $q_{opt}=(q_{1},q_{2},q_{3},q_{4})$& Optimal $TH_{S}$\\
\hline \hline
1 & $(0.4755,0.4755,1,0)$ & $0.6702$\\
2 &$(0.4259,0.0974,0.7478,0.6007)$ & $0.4958$\\
5 & $(0.2721,0.0415,0.1334,0.0512)$ & $0.2484$\\
15&$(0.0104,0.0053,0.4011,0.0143)$&$0.0936$\\
\hline
\end{tabular}
\caption{Optimal values of joining probabilities as functions of $M$.}\label{on1}
\end{table}
Under the same setting, but now fixing $N=4$, and by varying $M$, we observe in Table \ref{on1} that the optimal admission probabilities are more sensitive on the number of service phases compared with the number of retrieval phases. In particular, if the number of service phases increases, then the maximum throughput decreases very fast.

In Table \ref{nont} we derive optimal joining probabilities by excluding the constraints $q_{2}<q_{1}$, $q_{4}<q_{3}$. We can now observe that there is no specific trend on the values of the joining probabilities as in Table \ref{on}, while most of the cases $q_{2}>q_{1}$ and $q_{4}>q_{3}$. Similarly, to the case in Table \ref{on}, the number of phases of the retrieval times heavily affect the optimal values of the joining probabilities. 
\begin{table}
\centering
\begin{tabular}{|c| |c|c|}\hline
$N$ & $q_{opt}=(q_{1},q_{2},q_{3},q_{4})$& Optimal $TH_{S}$\\
\hline \hline
1 & $(1,0,0.7199,0)$ & $0.328$\\
2 &$(0.0933,0.7073,0.2362,0.5123)$ & $0.2908$\\
5 & $(0.0744,0.714,0.2016,0.5106)$ & $0.287$\\
15&$(0.179,0.9938,0.0001,0.7527)$&$0.295$\\
30&$(0.0293,00374,0.1065,0.0439)$&$0.276$\\
\hline
\end{tabular}
\caption{Optimal values of joining probabilities.}\label{nont}
\end{table}
\section{Conclusion \& future work}
In this work, we introduced the concept of \textit{event-dependent} arrivals in the retrial setting. We studied the stability condition, and investigated the stationary behaviour both at service completion, and at an arbitrary epoch. Explicit expressions for various performance metrics are derived, and used to numerically investigate the effect of event-dependency on system's performance. It seems that event-dependency is a result of customer's strategic behaviour. We perform constrained optimization problems that shown insights into the effect of event dependency on the admission control problem.

Several questions are open for future research. For instance, to formaly investigate the admission control problem and determine optimal policies. Another option is to investigate the possibility of obtaining the performance metrics through the QMCD (Queueing \& Markov Chain Decomposition) method \cite{baron}. Moreover, it would be worth investigating the effect of event-dependency on retrial and/or service times. Furthermore, to consider the case of linear retrial policy, which also allows blocked customers to retry for service.
\section*{Declaration of Competing Interest}
The author declares that he has no known competing financial interests or personal relationships that could have appeared to influence the work reported in this paper.
\appendix
\section{Proof of Theorem \ref{erg}}\label{ap1}(Sufficiency) We use standard Foster-Lyapunov arguments. The mean
drifts are given by:
\begin{displaymath}
\begin{array}{rl}
x_{n}=&E(X_{i+1}-X_{i}|X_{i}=n)=E(A_{i+1}(B_{i+1})|X_{i}=n)-E(B_{i+1}|X_{i}=n)\vspace{2mm}\\
=&(1-\alpha^{*}(\lambda^{-})(1-\delta_{0n}))A_{e}^{\prime}(1)+\alpha^{*}(\lambda^{-})(1-\delta_{0n}))(A_{r}^{\prime}(1)-1),
\end{array}
\end{displaymath} 
where $\delta_{0n}$ denotes Kronecker’s delta and
\begin{displaymath}\begin{array}{c}
A_{k}^{\prime}(1)=\frac{d}{dz}A_{k}(z)|_{z=1}=\sum_{n=0}^{\infty}nb_{n}^{k},\,k=pr,re.
\end{array}
\end{displaymath}Suppose that 
\begin{displaymath}
\begin{array}{c}
(1-\alpha^{*}(\lambda^{-}))A_{e}^{\prime}(1)+\alpha^{*}(\lambda^{-})A_{r}^{\prime}(1)<\alpha^{*}(\lambda^{-}).
\end{array}
\end{displaymath}
Then,
\begin{displaymath}
\begin{array}{c}
\epsilon=\frac{1}{2}[\alpha^{*}(\lambda^{-})-(1-\alpha^{*}(\lambda^{-}))A_{e}^{\prime}(1)-\alpha^{*}(\lambda^{-})A_{r}^{\prime}(1)]>0,
\end{array}
\end{displaymath}
and there exists
\begin{displaymath}
\begin{array}{rl}
\lim_{n\to\infty}x_{n}=&(1-\alpha^{*}(\lambda^{-}))A_{e}^{\prime}(1)+\alpha^{*}(\lambda^{-})A_{r}^{\prime}(1)-\alpha^{*}(\lambda^{-})\\
=&-2\epsilon<-\epsilon.
\end{array}
\end{displaymath}
Hence, $x_{n}<-\epsilon$ for
all the states except a finite number. Therefore,
\begin{equation}
\begin{array}{c}
(1-\alpha^{*}(\lambda^{-}))A_{e}^{\prime}(1)+\alpha^{*}(\lambda^{-})A_{r}^{\prime}(1)<\alpha^{*}(\lambda^{-}),
\end{array}\label{ergd}
\end{equation}
is a sufficient condition for the ergodicity of the embedded MC. After straightforward computations \eqref{ergd} reduces to \eqref{ergo}.\\
\textbf{(Necessity)} The necessity part can be proved using the Kaplan's condition and further details are omitted (The necessity can also proved using the generating function approach; see Theorem \ref{basic}).

\section{Proof of Lemma \ref{th3}}\label{ap2}
Considering the evolution of $\{(C(t),X(t),I(t),Z(t));t\geq 0\}$ in the interval $[0,t+dt]$ and conditioning on its value at time $t$ we have for $dt\to 0^{+}$, the equations
\begin{eqnarray}
\frac{d}{dt}p_{0,0}(t)=-\lambda^{-}p_{0,0}(t)+\sum_{k=2}^{3}p_{1,0}^{(k)}(0,t),\label{y1}\\
(\frac{\partial}{\partial t}-\frac{\partial}{\partial r})p_{0,j}(r,t)=-\lambda^{-}p_{0,j}(r,t)+a(r)\sum_{k=2}^{7}p_{1,j}^{(k)}(0,t),\,j\geq 1,\\
(\frac{\partial}{\partial t}-\frac{\partial}{\partial r})p_{1,j}^{(2)}(r,t)=-\lambda^{e}p_{1,j}^{(2)}(r,t)+\lambda^{-}p_{0,j}(t)b(r),\,j\geq 0,\\
(\frac{\partial}{\partial t}-\frac{\partial}{\partial r})p_{1,j}^{(3)}(r,t)=-\lambda^{r}p_{1,j}^{(3)}(r,t)+p_{0,j+1}(0,t)b(r),\,j\geq 0,\\
(\frac{\partial}{\partial t}-\frac{\partial}{\partial r})p_{1,j}^{(4)}(r,t)=-\lambda_{+}^{e}p_{1,j}^{(4)}(r,t)+\lambda^{e}p_{1,j-1}^{(2)}(r,t),\,j\geq 1,\\
(\frac{\partial}{\partial t}-\frac{\partial}{\partial r})p_{1,j}^{(5)}(r,t)=-\lambda_{+}^{e}p_{1,j}^{(5)}(r,t)+\lambda_{+}^{e}(p_{1,j-1}^{(5)}(r,t)+p_{1,j-1}^{(4)}(r,t)),\,j\geq 2,\\
(\frac{\partial}{\partial t}-\frac{\partial}{\partial r})p_{1,j}^{(6)}(r,t)=-\lambda_{+}^{r}p_{1,j}^{(6)}(r,t)+\lambda^{r}p_{1,j-1}^{(3)}(r,t),\,j\geq 1,\\
(\frac{\partial}{\partial t}-\frac{\partial}{\partial r})p_{1,j}^{(7)}(r,t)=-\lambda_{+}^{r}p_{1,j}^{(7)}(r,t)+\lambda_{+}^{r}(p_{1,j-1}^{(7)}(r,t)+p_{1,j-1}^{(6)}(r,t)),\,j\geq 2\label{y8}.
\end{eqnarray}
Letting $t\to\infty$, equations \eqref{y1}-\eqref{y8} reduce to those given in \eqref{bnm}.
\section{Proof of Theorem \ref{basic}}\label{ap3}
Multiplying the second in \eqref{bnm} with $e^{-sr}$, integrating with respect to $s\in[0,\infty)$, and having in mind that $p_{0,j}^{*}(s)=\int_{0}^{\infty}e^{-sr}p_{0,j}(r)dr$,  yields
\begin{displaymath}
\begin{array}{rl}
(s-\lambda^{-})p_{0,j}^{*}(s)=p_{0,j}(0)-\alpha^{*}(s)\sum_{k=2}^{7}p_{1,j}^{(k)}(0).
\end{array}
\end{displaymath}
Multiplying with $z^{j}$, and summing for all $j\geq 0$ yields
\begin{equation}
\begin{array}{rl}
(s-\lambda^{-})P_{0}^{*}(s,z)=P_{0}(0,z)-\alpha^{*}(s)[\sum_{k=2}^{7}P_{1}^{(k)}(0,z)-\lambda^{-}p_{0,0}].
\end{array}\label{xc}
\end{equation}
For $s=\lambda^{-}$ we obtain, 
\begin{equation}
\begin{array}{c}
P_{0}(0,z)=\alpha^{*}(\lambda^{-})[\sum_{k=2}^{7}P_{1}^{(k)}(0,z)-\lambda^{-}p_{0,0}],
\end{array}\label{bv}
\end{equation}
and substituting back in \eqref{xc},
\begin{equation}
\begin{array}{c}
P_{0}^{*}(s,z)=\frac{\alpha^{*}(\lambda^{-})-\alpha^{*}(s)}{s-\lambda^{-}}[\sum_{k=2}^{7}P_{1}^{(k)}(0,z)-\lambda^{-}p_{0,0}].
\end{array}\label{b1}
\end{equation}
By repeating the same procedure for the third in \eqref{bnm}, we obtain
\begin{equation}
\begin{array}{c}
(s-\lambda^{e})P_{1,2}^{*}(s,z)=P_{1,2}(0,z)-\lambda^{-}\beta^{*}(s)[p_{0,0}+P_{0}^{*}(0,z)],
\end{array}\label{b02}
\end{equation}
and setting $s=\lambda^{e}$ yields
\begin{displaymath}
P_{1,2}(0,z)=\lambda^{-}\beta^{*}(\lambda^{e})[p_{0,0}+P_{0}^{*}(0,z)].
\end{displaymath}
Substituting back in \eqref{b02} yields
\begin{equation}
\begin{array}{c}
P_{1,2}^{*}(s,z)=\frac{\lambda^{-}}{s-\lambda^{e}}(\beta^{*}(\lambda^{e})-\beta^{*}(s))[p_{0,0}+P_{0}^{*}(0,z)].
\end{array}\label{b2}
\end{equation}
So the second in \eqref{t1s1} has been proved. By applying the same procedure for the forth in \eqref{bnm} we obtain
\begin{equation}
\begin{array}{rl}
P_{1,3}(0,z)=&\frac{\beta^{*}(\lambda^{r})}{z}P_{0}(0,z),\\
P_{1,3}^{*}(s,z)=&\frac{\beta^{*}(\lambda^{r})-\beta^{*}(s)}{z(s-\lambda^{r})}P_{0}(0,z).
\end{array}\label{b3}
\end{equation}
Now repeat the same procedure for the fifth and sixth in \eqref{bnm} we obtain
\begin{displaymath}
\begin{array}{rl}
(s-\lambda_{+}^{e})P_{1,4}^{*}(s,z)=&P_{1,4}(0,z)-\lambda^{e}zP_{1,2}^{*}(s,z),\\
(s-\lambda_{+}^{e})P_{1,5}^{*}(s,z)=&P_{1,5}(0,z)-\lambda_{+}^{e}z(P_{1,5}^{*}(s,z)+P_{1,4}^{*}(s,z)).
\end{array}
\end{displaymath}
Summing the above equations and using \eqref{b2}, we obtain
\begin{equation}
\begin{array}{rl}
\sum_{k=4}^{5}P_{1,k}(0,z)=&\frac{\lambda^{-}\lambda^{e}z[\beta^{*}(\lambda^{e})-\beta^{*}(\lambda^{e}(1-z))]}{\lambda_{+}^{e}(1-z)-\lambda^{e}}[p_{0,0}+P_{0}^{*}(0,z)],\\
\sum_{k=4}^{5}P_{1,k}^{*}(s,z)=&\frac{\lambda^{-}\lambda^{e}z}{s-\lambda_{+}^{e}(1-z)}\left(\frac{\beta^{*}(\lambda^{e})-\beta^{*}(\lambda^{e}(1-z))}{\lambda_{+}^{e}(1-z)-\lambda^{e}}-\frac{\beta^{*}(\lambda^{e})-\beta^{*}(s)}{s-\lambda^{e}}\right)\\&\times[p_{0,0}+P_{0}^{*}(0,z)].
\end{array}\label{b4}
\end{equation}
Similar operations for the last two equations in \eqref{bnm} yield
\begin{equation}
\begin{array}{rl}
\sum_{k=6}^{7}P_{1,k}(0,z)=&\frac{\lambda^{r}[\beta^{*}(\lambda^{r})-\beta^{*}(\lambda^{r}(1-z))]}{\lambda_{+}^{r}(1-z)-\lambda^{r}}P_{0}(0,z),\\
\sum_{k=6}^{7}P_{1,k}^{*}(s,z)=&\frac{\lambda^{r}P_{0}(0,z)}{s-\lambda_{+}^{r}(1-z)}\left(\frac{\beta^{*}(\lambda^{r})-\beta^{*}(\lambda^{r}(1-z))}{\lambda_{+}^{r}(1-z)-\lambda^{r}}-\frac{\beta^{*}(\lambda^{r})-\beta^{*}(s)}{s-\lambda^{r}}\right).
\end{array}\label{b5}
\end{equation}
Now from \eqref{b1},
\begin{equation}
\sum_{k=2}^{7}P_{1}^{(k)}(0,z)=\lambda^{-}p_{0,0}+\frac{P_{0}(0,z)}{\alpha^{*}(\lambda^{-})}.\label{v1}
\end{equation}
Using \eqref{b2}-\eqref{b5}, and after lengthy algebraic computations, we obtain
\begin{equation}
\sum_{k=2}^{7}P_{1}^{(k)}(0,z)=\lambda^{-}A_{pr}(z)(p_{0,0}+P_{0}^{*}(0,z))+\frac{P_{0}(0,z)}{z}A_{re}(z).\label{v2}
\end{equation}
Equating \eqref{v1}, \eqref{v2} we obtain
\begin{equation}
P_{0}(0,z)=\frac{\lambda^{-}z\alpha^{*}(\lambda^{-})(p_{0,0}(A_{pr}(z)-1)+A_{re}(z)P_{0}^{*}(0,z))}{z-\alpha^{*}(\lambda^{-})A_{re}(z)}.\label{gb}
\end{equation}
Using \eqref{bv}, \eqref{gb}, we obtain
\begin{equation}
\begin{array}{c}
P_{0}^{*}(s,z)=\frac{\alpha^{*}(\lambda^{-})-\alpha^{*}(s)}{s-\lambda^{-}}\frac{P_{0}(0,z)}{\alpha^{*}(\lambda^{-})}
\end{array}\label{z1}
\end{equation}
Setting $s=0$ in \eqref{z1}, and using \eqref{gb} we obtain the first in \eqref{t1s1}. Setting $s=0$ in \eqref{b3}, and letting $K(z):=\frac{P_{0}(0,z)}{z}$, we obtain the third in \eqref{t1s1}. Similarly we can obtain the rest expressions in \eqref{t1s1}.

Having obtain the expressions in \eqref{t1s1}, and having in mind that $1=p_{0,0}+P_{0}^{*}(0,1)+\sum_{k=2}^{7}P_{1,k}^{*}(0,1)$ we derive after lengthy but straightforward computations the probability of an empty system $p_{0,0}$. 
\section{Proof of Corollary \ref{cor1}}\label{ap4}
Let $A_{k}^{(j)}(1)=\frac{d^{j}}{dz^{j}}A_{k}(z)|_{z=1}$, $k=e,r$, $j=1,2$. Then, for $k=e,r$ we have
\begin{displaymath}
\begin{array}{rl}
A_{k}^{(1)}(1)=&\frac{(\lambda^{k}-\lambda_{+}^{k})(1-\beta^{*}(\lambda^{k}))+\lambda^{k}\lambda_{+}^{k}\bar{b}}{\lambda^{k}},\\
A_{k}^{(2)}(1)=&\lambda_{+}^{k}(2\bar{b}+\lambda_{+}^{k}\bar{b}^{(2)})-2\frac{\lambda_{+}^{k}}{\lambda^{k}}A_{k}^{(1)}(1).
\end{array}
\end{displaymath}
Let also,
\begin{displaymath}
\begin{array}{rl}
G:=&\frac{\alpha^{*}(\lambda^{-})[(2A_{e}^{(1)}(1)+A_{e}^{(2)}(1))(1-A_{r}^{(1)}(1))+A_{e}^{(1)}(1)A_{r}^{(2)}(1)]}{2[A_{e}^{(1)}(1)(\alpha^{*}(\lambda^{-})-1)+\alpha^{*}(\lambda^{-})(1-A_{r}^{(1)}(1))]^{2}},\vspace{2mm}\\
F:=&\frac{[p_{0,0}(1-\alpha^{*}(\lambda^{-}))G+A_{e}^{(1)}(1)P(C=0)]-P_{0}^{*}(0,1)(1-\alpha^{*}(\lambda^{-}A_{r}^{(1)}(1)))}{(1-\alpha^{*}(\lambda^{-}))^{2}}.
\end{array}
\end{displaymath}
Then, using the results in Theorem \ref{basic}, and differentiating with respect to $z$, at $z=1$, we obtain after heavy computations
\begin{equation}
\begin{array}{rl}
\frac{d}{dz}P_{0}^{*}(0,z)|_{z=1}=&p_{0,0}(1-\alpha^{*}(\lambda^{-}))G,\\
\frac{d}{dz}P_{1,2}^{*}(0,z)|_{z=1}=&p_{0,0}(1-\alpha^{*}(\lambda^{-}))\lambda^{-}\frac{1-\beta^{*}(\lambda^{e})}{\lambda^{e}}G,\\
\frac{d}{dz}P_{1,3}^{*}(0,z)|_{z=1}=&\lambda^{-}\alpha^{*}(\lambda^{-})\frac{1-\beta^{*}(\lambda^{r})}{\lambda^{r}}F,\\
\frac{d}{dz}\sum_{k=4}^{5}P_{1,k}^{*}(0,z)|_{z=1}=&\lambda^{-}(\bar{b}-\frac{1-\beta^{*}(\lambda^{e})}{\lambda^{e}})[\frac{d}{dz}P_{0}^{*}(0,z)|_{z=1}+\frac{\lambda^{e}-\lambda_{+}^{e}}{\lambda^{e}}P(C=0)]\\&+\frac{\lambda^{-}\lambda_{+}^{e}\bar{b}^{(2)}}{2}P(C=0),\\
\frac{d}{dz}\sum_{k=6}^{7}P_{1,k}^{*}(0,z)|_{z=1}=&(\bar{b}-\frac{1-\beta^{*}(\lambda^{r})}{\lambda^{r}})(\frac{\lambda^{r}-\lambda_{+}^{r}}{\lambda^{e}}P_{0}(0,1)+F)+\frac{\lambda_{+}^{r}\bar{b}^{(2)}}{2}P_{0}(0,1).
\end{array}\label{opa}
\end{equation}
Sum the terms in \eqref{opa} to obtain $E(X)$ in \eqref{metr}. Let $TH_{S}$ is the throughput generated by the system. Then,
\begin{displaymath}
\begin{array}{rl}
TH_{S}=&\lambda^{-}P_{0}^{*}(0,1)+\lambda^{e}P_{1,2}^{*}(0,1)+\lambda_{+}^{e}\sum_{k=4}^{5}P_{1,k}^{*}(0,1)\\
&+\lambda^{r}P_{1,3}^{*}(0,1)+\lambda_{+}^{r}\sum_{k=6}^{7}P_{1,k}^{*}(0,1).
\end{array}
\end{displaymath}
After tedious computations we can have the expression given in \eqref{metr}. 
\bibliographystyle{unsrt}

\bibliography{arxiv}

\begin{thebibliography}{10}

\bibitem{call1}
Mor Armony and Constantinos Maglaras.
\newblock On customer contact centers with a call-back option: Customer
  decisions, routing rules, and system design.
\newblock {\em Operations Research}, 52(2):271--292, 2004.

\bibitem{call2}
Mor Armony and Constantinos Maglaras.
\newblock Contact centers with a call-back option and real-time delay
  information.
\newblock {\em Operations Research}, 52(4):527--545, 2004.

\bibitem{dud1}
Alexander~N Dudin, Achyutha Krishnamoorthy, VC~Joshua, and Gennady~V Tsarenkov.
\newblock Analysis of the {BMAP/G/1} retrial system with search of customers
  from the orbit.
\newblock {\em European Journal of Operational Research}, 157(1):169--179,
  2004.

\bibitem{phung}
Tuan Phung-Duc.
\newblock {\em Retrial Queueing Models: A Survey on Theory and Applications},
  pages 1--31.
\newblock Stochastic Operations Research in Business and Industry. World
  Scientific, 05 2017.

\bibitem{gencer}
Busra Gencer, Zeynep Karaesmen, Evrim Gunes, and Ozge Pala.
\newblock The impact of queue features on customers' queue joining and reneging
  behavior: A laboratory experiment.
\newblock IFORS 2014, 07 2014.

\bibitem{falin}
Gennadi Falin and James~GC Templeton.
\newblock {\em Retrial queues}, volume~75.
\newblock CRC Press, 1997.

\bibitem{arta}
Jes{\`u}s~R.. Artalejo and Antonio G{\'o}mez-Corral.
\newblock {\em Retrial Queueing Systems: A Computational Approach}.
\newblock Springer-Verlag Berlin Heidelberg, 2008.

\bibitem{dim1}
Christos Langaris and Ioannis Dimitriou.
\newblock A queueing system with n-phases of service and (n-1)-types of retrial
  customers.
\newblock {\em European Journal of Operational Research}, 205(3):638--649,
  2010.

\bibitem{dim2}
Ioannis Dimitriou.
\newblock A two-class queueing system with constant retrial policy and general
  class dependent service times.
\newblock {\em European Journal of Operational Research}, 270(3):1063--1073,
  2018.

\bibitem{fayo}
G~Fayolle.
\newblock A simple telephone exchange with delayed feedbacks.
\newblock In {\em Proc. of the international seminar on Teletraffic analysis
  and computer performance evaluation}, pages 245--253, 1986.

\bibitem{choi}
B.D. Choi, K.K. Park, and C.E.M. Pearce.
\newblock An {M/M/1} retrial queue with control policy and general retrial
  times.
\newblock {\em Queueing Systems}, 14(3-4):275 – 292, 1993.
\newblock Cited by: 51.

\bibitem{gomez}
Antonio G{\'o}mez-Corral.
\newblock Stochastic analysis of a single server retrial queue with general
  retrial times.
\newblock {\em Naval Research Logistics (NRL)}, 46(5):561--581, 1999.

\bibitem{baron}
Opher Baron, Antonis Economou, and Athanasia Manou.
\newblock The state-dependent {M/G/1} queue with orbit.
\newblock {\em Queueing Systems}, 90(1):89--123, 2018.

\bibitem{legros2018}
Benjamin Legros.
\newblock {M/G/1} queue with event-dependent arrival rates.
\newblock {\em Queueing Systems}, 89(3):269--301, 2018.

\bibitem{legros2018q}
Benjamin Legros and Ali~Devin Sezer.
\newblock Stationary analysis of a single queue with remaining service
  time-dependent arrivals.
\newblock {\em Queueing Systems}, 88(1):139--165, 2018.

\bibitem{legros2021}
Benjamin Legros.
\newblock Dimensioning a queue with state-dependent arrival rates.
\newblock {\em Computers \& Operations Research}, 128:105179, 2021.

\bibitem{legros2022}
Benjamin Legros.
\newblock The principal-agent problem for service rate event-dependency.
\newblock {\em European Journal of Operational Research}, 297(3):949--963,
  2022.

\bibitem{beker}
René Bekker, S.~Borst, O.J. Boxma, and Offer Kella.
\newblock Queues with workload-dependent arrival and service rates: In honor of
  vladimir kalashnikov (guest editors: Evsey morozov and richard serfozo).
\newblock {\em Queueing Systems}, 46, 03 2004.

\bibitem{box}
Onno Boxma, Haya Kaspi, Offer Kella, and David Perry.
\newblock On/off storage systems with state-dependent input, output, and
  switching rates.
\newblock {\em Probability in the Engineering and Informational Sciences},
  19(1):1–14, 2005.

\bibitem{kern}
Yoav Kerner.
\newblock The conditional distribution of the residual service time in the {M n
  /G/1} queue.
\newblock {\em Stochastic Models}, 24(3):364--375, 2008.

\bibitem{box2}
Onno~J Boxma and Maria Vlasiou.
\newblock On queues with service and interarrival times depending on waiting
  times.
\newblock {\em Queueing Systems}, 56(3):121--132, 2007.

\bibitem{dau}
Bernardo D’Auria, Ivo~J.B.F. Adan, René Bekker, and Vidyadhar Kulkarni.
\newblock An {M/M/c} queue with queueing-time dependent service rates.
\newblock {\em European Journal of Operational Research}, 299(2):566--579,
  2022.

\bibitem{artatop}
JR~Artalejo and JI~Falin.
\newblock Stochastic decomposition for retrial queues.
\newblock {\em Top}, 2(2):329--342, 1994.

\end{thebibliography}

\end{document}